\documentclass[11pt,reqno,tbtags]{amsart}

\usepackage{amsthm}
\usepackage{amssymb}
\usepackage{amsfonts, mathrsfs}
\usepackage{amsmath}
\usepackage[numbers, sort&compress]{natbib}
\usepackage{graphicx}
\usepackage{vmargin, enumerate}
\usepackage{setspace}
\usepackage{paralist}
\usepackage{hyperref}
\usepackage{color}

\providecommand{\R}{}
\providecommand{\Z}{}
\providecommand{\N}{}

\renewcommand{\R}{\mathbb{R}}
\renewcommand{\Z}{\mathbb{Z}}
\renewcommand{\N}{{\mathbb N}}

\newcommand{\p}[1]{{\mathbf P}\left\{#1\right\}}

\newcommand{\I}[1]{{\mathbf 1}_{[#1]}}

 \newcommand{\bag}{\begin{align}}
\newcommand{\bags}{\begin{align*}}
\newcommand{\eag}{\end{align*}}
\newcommand{\eags}{\end{align*}}

\newtheorem{thm}{Theorem}

\newtheorem{prop}[thm]{Proposition}
\newtheorem{cor}[thm]{Corollary}

\newcommand\cB{\mathcal B}
\newcommand\cC{\mathcal C}

\newcommand\cM{\mathcal M}

\newcommand\cS{{\mathcal S}}
\newcommand\cT{{\mathcal T}}

\hypersetup{
    bookmarks=true,         
    unicode=false,          
    pdftoolbar=true,        
    pdfmenubar=true,        
    pdffitwindow=true,      
    pdftitle={My title},    
    pdfauthor={Author},     
    pdfsubject={Subject},   
    pdfnewwindow=true,      
    pdfkeywords={keywords}, 
    colorlinks=true,       
    linkcolor=blue,          
    citecolor=blue,        
    filecolor=blue,      
    urlcolor=blue           
}

\reversemarginpar
\definecolor{clou}{rgb}{0.5,0.125,0.3125}

\newcommand{\eps}{\epsilon}

\newcommand\urladdrx[1]{{\urladdr{\def~{{\tiny$\sim$}}#1}}}

\begingroup
  \divide\count255 by 60
  \multiply\count255 by -60
  \advance\count255 by \time
  \ifnum \count255 < 10 \xdef\oclock{\the\count1:0\the\count255}
  \else\xdef\oclock{\the\count1:\the\count255}\fi
\endgroup

\usepackage{caption,subcaption}

\usepackage{wrapfig}

\newcommand{\rT}{\mathrm{T}}
\newcommand{\rM}{\mathrm{M}}
\newcommand{\rQ}{\mathrm{Q}}
\newcommand{\rG}{\mathrm{G}}
\newcommand{\rP}{\mathrm{P}}
\usepackage{stmaryrd}

\newcommand{\lrb}[1]{\llbracket #1 \rrbracket}
\newcommand{\ol}[1]{#1^{\to}}
\providecommand{\v}{}
\renewcommand{\v}{\mathrm{v}}
\newcommand{\parent}{\alpha}
\begin{document}

\title[Growing random maps]{Growing random 3-connected maps\\
{\footnotesize or { \em Comment s'enfuir de l'hexagone}}}
\author{L. Addario-Berry}
\address{Department of Mathematics and Statistics, McGill University, 805 Sherbrooke Street West, 
		Montr\'eal, Qu\'ebec, H3A 2K6, Canada}
\email{louigi@math.mcgill.ca}
\date{February 10, 2014} 
\urladdrx{http://www.math.mcgill.ca/~louigi/}


\begin{abstract} 
We use a growth procedure for binary trees \cite{luczak04building}, a bijection between binary trees and irreducible quadrangulations of the hexagon \cite{fusy08dissection}, and the classical angular mapping between quadrangulations and maps, to define a growth procedure for maps. The growth procedure is local, in that every map is obtained from its predecessor by an operation that only modifies vertices lying on a common face with some fixed vertex. As $n \to \infty$, the probability that the $n$'th map in the sequence is 3-connected tends to $2^8/3^6$. 
The sequence of maps has an almost sure limit $G_{\infty}$, and we show that $G_{\infty}$ is the {\em distributional} local limit of large, uniformly random 3-connected graphs. 
\end{abstract}
\maketitle


\section{\large {\bf Introduction}}\label{sec:intro} 
Here is a common situation in probability. We have found a result of the form "as $n \to \infty$, $X_n \to X_{\infty}$ in distribution", where $(X_n,1 \le n \le \infty)$ are some sort of random objects. The Skorohod embedding theorem then guarantees (in great generality) the existence of a coupling such that $X_n \to X_{\infty}$ almost surely. This result, though useful, is existential, and often the discovery of an {\em explicit} coupling leads to a deeper understanding of both the limit object and its finite approximations.  

Given the substantial recent interest in random planar maps having the Brownian map as their (known or conjectural) scaling limit, it seems natural to seek such a coupling for random maps. (We hereafter refer to such a coupling as a {\em growth procedure}.) To date, however, all convergence results for such random planar maps have been distributional in nature. The goal of this note is to provide an explicit, local -- in the sense described in the abstract -- growth procedure for random planar maps. The result is a sequence of rooted maps $(\mathrm{M}_n, 1 \le n \le \infty)$, such that $\mathrm{M}_n \stackrel{\mathrm{a.s.}}{\rightarrow} \mathrm{M}_\infty$ as $n \to \infty$, by which we mean that balls of any fixed radius around the root almost surely stabilize. 

We briefly summarize the structure and the arguments of the paper, then head right to details. We begin by considering {\em irreducible quadrangulations of the hexagon}: these are rooted maps with a single face of degree 6 and all other faces of degree 4, such that every cycle of length 4 bounds a face. (The root is an uniformly random oriented edge; it need not lie along the hexagonal face.) Such maps are in bijective correspondence with rooted binary plane trees \cite{fusy08dissection}; we describe the bijection in Section~\ref{sec:closurebij}. In Section~\ref{sec:budgrowth-mapgrowth} we describe how growing a binary tree -- transforming a degree-one vertex into a degree-three vertex -- changes the corresponding irreducible quadrangulation of the hexagon.

A growth procedure is already known to exist for random binary plane trees \cite{luczak04building}; the almost sure limit $\mathrm{T}_{\infty}$ is a critical binomial Galton-Watson tree, conditioned to survive. In Section~\ref{sec:growing} we use the bijection and properties of $\mathrm{T}_{\infty}$ to show that the corresponding growth procedure for irreducible quadrangulations of the hexagon has an almost sure limit. We then show, in Section~\ref{sec:3-connectedmaps}, that the structure of the irreducible quadrangulations near their root asymptotically decouples from the structure near the hexagonal face. We use this fact together with the angular mapping between quadrangulations and general maps to define a growth procedure $(G_n,1 \le n \le \infty)$ for random maps. We show that the almost sure limit $G_{\infty}$ is almost surely $3$-connected and is the distributional limit of large random $3$-connected maps. We conclude in Section~\ref{sec:conc} with a number of remarks and suggestions for future research. 

\section{\large {\bf Definitions}}
We briefly review some basic concepts regarding (planar) maps; more details can be found in, e.g., \cite{lando04graphs}. All our maps are connected. Also, maps are by default embedded in $\R^2$ rather than on the sphere $\mathbb{S}^2$. For any graph or map $M$, we write $v(M)$ and $e(M)$ for the nodes and edges of $M$, respectively. The {\em corners} incident to a face $f$ of $M$ are the angles $\kappa=(\{u,v\},\{v,w\})$ formed by consecutive edges along the face. We say $\kappa$ is incident to $f$, to $v$, and to its constituent edges, and write $\v(\kappa)=v$. The degree of $f$ is the number of corners incident to $f$. Write $\cC(M)$ for the set of corners of $M$. 

For $\{u,v\} \in e(M)$, we use both $(u,v)$ and $uv$ as notation for the orientation of $\{u,v\}$ with tail $u$ and head $v$. An oriented circuit $C$ in $M$ is {\em clockwise} if the bounded region of $\R^2\setminus C$ lies to the right of $C$, and is otherwise counterclockwise. A {\em rooted map} is a pair $\rM=(M,vw)$, where $\{v,w\} \in e(M)$; $vw$ is the {\em root edge}. The {\em root corner} $\rho=\rho(\rM)$ is the unique corner  incident to $v$ whose second incident edge is $\{v,w\}$, and the {\em root node} is $r(\rM)=u$. 

Given a rooted map $\rM=(M,uv)$ and $R > 0$, write $\rM^{< R}$ for the submap induced by the set of nodes at graph distance less than $R$ from $u$. We say a sequence $(\rM_n,n \ge 1)$ of (finite or infinite) locally finite rooted maps converges to rooted map $\rM_\infty$, and write $\rM_n \to \rM_{\infty}$, if for all $R > 0$, there is $n_0 \in \N$ such that for all $n \ge n_0$, $\rM_n^{< R}$ and $\rM_{\infty}^{< R}$ are isomorphic as rooted maps. Note that since the maps $\rM_n$ are locally finite, so is their limit $\rM_{\infty}$. This is often called {\em local weak convergence}; see \cite{benjamini01recurrence,aldous04objective,gg13recurrence} for more details. 

A {\em quadrangulation} is a map in which every face has degree four. In a {\em quadrangulation of the hexagon}, the unbounded face has degree six and all others have degree four. A quadrangulation, or a quadrangulation of the hexagon, is called {\em irreducible} if every cycle of length four bounds a face. It is easily seen that an irreducible quadrangulation of the hexagon is necessarily simple, and also $3$-connected, so has an unique embedding by Whitney's theorem. 

\begin{wrapfigure}[14]{R}{.3\textwidth}
\vspace{-0.3cm}
\rule[5pt]{.3\textwidth}{0.5pt}
                 \includegraphics[width=.3\textwidth,page=1]{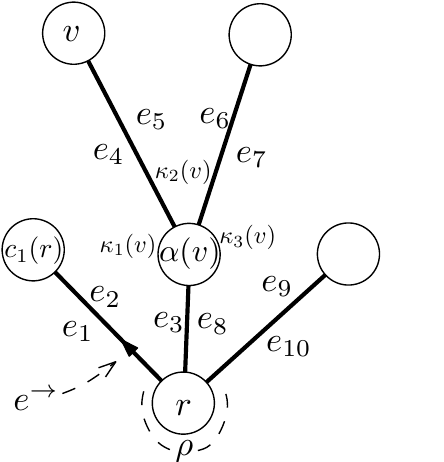}
\caption{\small A binary (plane) tree $(T,r)$. With labels as in the figure, $\kappa_3(v) \prec \rho \prec \kappa_1(v)$. }
\label{fig:bintree}
\rule[8pt]{.3\textwidth}{0.5pt} 
\end{wrapfigure}
The definitions of the coming paragraphs are illustrated in Figure~\ref{fig:bintree}. 
A {\em plane tree} is a rooted tree $\rT=(T,r(T))$ together with an ordering of the children of each node $v$ as $c_1(v),\ldots,c_{k_{\rT}(v)}$, where $k_{\rT}(v)$ is the number of children of $v$. This collection of orderings uniquely specifies $\rT$ as a planar map, with root $e^{\to} = e^{\to}(\rT)=(r(T),c_1(r(T)))$. 
Conversely, the orderings may be recovered from the embedding and the root edge $e^{\to}(\rT)$. 

Viewed as a map, $\rT$ has an unique (unbounded) face. Writing $n=|v(T)|$, list the (oriented) edges of this face in the order they are traversed by a counterclockwise tour\footnote{Recall that such a tour keeps the unbounded face to its {\em left}.} of $\rT$ starting from 
$e^{\to}(\rT)$, 
as $\ol{e}_1(\rT),\dots,\ol{e}_{2n-2}(\rT)$, or simply $\ol{e}_1,\ldots,\ol{e}_{2n-2}$ when the tree is clear from context. Write $e_1,\ldots,e_{2n-2}$ for the corresponding unoriented edges. Note that each edge of $\rT$ appears exactly twice in this sequence. It is also convenient to set $\ol{e}_{2n-1}=\ol{e}_1$. We then have $\rho(\rT)=\{e_{2n-2},e_1\}$. 

For $v \in v(\rT)$, list the corners incident to $v$ as $\kappa_1(v),\ldots,\kappa_{d}(v)$, in the order they appear in the counterclockwise tour. Also, write $\prec$ for the cyclic order on corners induced by the counterclockwise tour. In other words, $\kappa \prec \kappa' \prec \kappa^*$ iff $\kappa' \not \in \{\kappa,\kappa^*\}$ and the cyclic tour starting from $\kappa$ visits $\kappa'$ before $\kappa^*$; in this case we say $\kappa^*$ is between $\kappa$ and $\kappa'$. Finally, for $v,w \in v(\rT)$, write $\lrb{v,w}$ for the unique simple path from $v$ to $w$ in $\rT$.

In this paper, a {\em binary tree} is a plane tree $\rT=(T,r(T))$ all of whose nodes have degree either one or three.
We call the degree one and three nodes of $\rT$ the {\em buds} and {\em internal nodes} of $\rT$, and denote them $B(\rT)$ and $I(\rT)$, respectively. Likewise, {\em bud corners} and {\em internal corners} have their obvious meanings, and we write $\cC_B(\rT)$ and $\cC_I(\rT)$ for the sets of bud and internal corners, respectively. We always have $|\cC_I(\rT)|/3=|I(\rT)|=|B(\rT)|-2=|\cC_B(\rT)|-2$. For a bud $v$, write $\parent(v)$ for the unique node of $\rT$ adjacent to $v$ ($\parent(v)$ is the parent of $v$ unless $v=r(T)$). 

\section{\large {\bf Bijections and growth procedures for trees and maps}}
Let $\rT=(T,r(T))$ be a binary tree, and write $e^{\to}=e^{\to}(\rT)=(r(T),c_1(r(T)))$ as above. In the first subsection, we describe a labelling of the corners of $\rT$, which we then use to define an irreducible quadrangulation of the hexagonal. The quadrangulation has a subtree of $\rT$ as a canonical ``nearly-spanning'' tree. This construction, due to Fusy, Poulalhon, and Schaeffer \cite{fusy08dissection}, is invertible and so bijective. In the second subsection we analyze the effect of growing a binary tree on the quadrangulation associated to it by the bijection we now describe.

\subsection{The Fusy-Poulalhon-Schaeffer ``closure'' bijection}\label{sec:closurebij}

Given a corner $\kappa \in \cC(T)$, for $i \in \{1,2,3\}$ let $N_i(\kappa)$ be the number of $i$'th children in 
$\lrb{r(T),\v(\kappa)}$. Note that all nodes except $r(T)$ have either zero or two children, and $r(T)$ has either one or three children.

If $\kappa \in \cC(\rT)$ is incident to node $v$ and is the $i$'th such corner, $\kappa_i(v)=\kappa$, then set $\eps_{\kappa}=2(i-1)$, so $\eps_{\kappa} \in \{0,2,4\}$. Then set 
\begin{equation}\label{eq:labeldef}
S_{\rT}(\kappa) = \begin{cases}
		3N_3(\kappa)+N_2(\kappa)-N_1(\kappa)+\eps_{\kappa} & \mbox{if}~\kappa \ne \rho \\
		 -4\I{\rho~\mathrm{is~a~bud~corner}} & \mbox{if}~\kappa=\rho
		 \end{cases}
\end{equation}

 \begin{wrapfigure}[16]{R}{.3\textwidth}
\rule[5pt]{.3\textwidth}{0.5pt}
                 \includegraphics[width=0.3\textwidth,page=1]{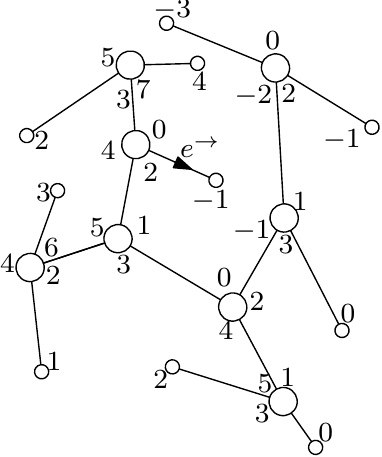}
\caption{\small A binary tree endowed with the corner labelling $S_{\rT}$.}
\label{fig:labeling}
\rule[8pt]{.3\textwidth}{0.5pt} 
\end{wrapfigure}
 An example appears in Figure~\ref{fig:labeling}. 
Here is another description of $S_{\rT}$, which is easily seen to be equivalent. 
 Give $\rho$ label $-4\I{\rho~\mathrm{is~a~bud~corner}}$, then perform a cyclic tour of the tree starting from $\rho$. When moving from an inner corner to another corner, decrease the label by one. When moving from a bud corner to another (necessarily inner) corner, increase the label by three. Finally, when returning to the root corner, subtract an additional six. 

Given a bud corner $\kappa$, if there exists an internal corner $\kappa'$ such that $S_{\rT}(\kappa') \le S_{\rT}(\kappa) - 6\I{\kappa \prec \rho \preceq \kappa'}$ then let $\sigma(\kappa)=\sigma_{\rT}(\kappa)$ be the first such corner $\kappa'$ (i.e., if $\kappa^*$ is another such internal corner then $\kappa \prec \kappa' \prec \kappa^*$). Call $\sigma(\kappa)$ the {\em attachment} corner of $\kappa$. Necessarily $S_{\rT}(\sigma(\kappa))=S_{\rT}(\kappa) - 6\I{\kappa \prec \rho \preceq \sigma(\kappa)}$ since labels decrease by at most one along edges (except at $\rho$, but this is accounted for by the correction term for winding around $\rho$). 

Let $\Delta=\Delta(\rT)$ be the set consisting of those corners $\kappa \in \cC_B(T)$ for which there is no $\kappa'$ with $S_{\rT}(\kappa') \le S_{\rT}(\kappa) - 6\I{\kappa \prec \rho \preceq \kappa'}$. 
Write $s^*=s^*(\rT) = \min\{S_{\rT}(\kappa): \kappa \in \cC(T)\}$. Then for each $\kappa \in \Delta$, $S_{\rT}(\kappa) \in [s^*,s^*+6]$, and we set $\sigma(\kappa)=s(\kappa)-s^*-6\I{s(\kappa)-s^*= 6}$. 

We now form a map from $\rT$ as follows. 
By {\em closing} a bud corner $\kappa \not \in \Delta$, we mean identifying $\v(\kappa)$ and $\v(\sigma(\kappa))$ to form an oriented edge $e_{\kappa}=(\parent(\v(\kappa)),\v(\sigma(\kappa)))$, in such a way that the oriented cycle $C_{\kappa}$ formed by following $e_{\kappa}$, then returning to $\parent(\v(\kappa))$ via $\lrb{\v(\sigma(\kappa)),\parent(\v(\kappa))}$, is clockwise (the bounded region of $\R^2 \setminus C$ lies to its right). 

Draw a hexagon so that $\rT$ lies in its bounded face (interior), and label its interior corners $0,1,\ldots,5$ so that for $\kappa \in \Delta$, $\sigma(\kappa)$ is a corner of the hexagon. (Later, we will also view the nodes of the hexagon as having labels $0,1,2,3,4,5$, in the obvious way.) For each bud corner $\kappa \not \in \Delta$, close $\kappa$ (see Figure~\ref{fig:closure-inner}). For each $\kappa \in \Delta$, identify $\v(\kappa)$ with the hexagon corner labelled $\sigma(\kappa)$ (see Figure~\ref{fig:closure-hex})

Writing $e^{\to}$ for the image of the root edge $e^{\to}(\rT)$ in $M$, 
the resulting map $\rM=(M,e^{\to})$ is an irreducible, edge-rooted quadrangulation of a hexagon. The key point of this section (Proposition~\ref{prop:bijection}, below) is that the function taking $\rT$ to $\rM$ is bijective \cite{fusy08dissection}. We use this fact more-or-less as a black box; however, the next two paragraphs contain a very brief sketch of one aspect of its proof, for the interested reader. A more detailed, very readable explanation can be found in \cite[Section 4]{fusy08dissection}.

\begin{figure}[ht]
\hspace{-0.5cm}
\begin{subfigure}[b]{0.3\textwidth}
		\includegraphics[width=\textwidth,page=2]{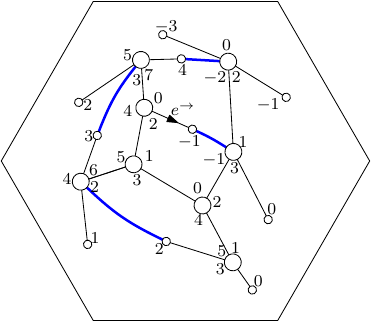}
                \caption{Blue edges denote closures to internal corners.\\}
                \label{fig:closure-inner}
\end{subfigure}%
\quad
\begin{subfigure}[b]{0.3\textwidth}
		\includegraphics[width=\textwidth,page=3]{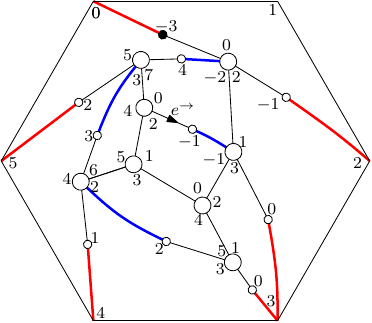}
                \caption{Red edges denote closures of corners in $\Delta$.\\}
                \label{fig:closure-hex}
\end{subfigure}%
\quad
\begin{subfigure}[b]{0.3\textwidth}
		\includegraphics[width=\textwidth,page=4]{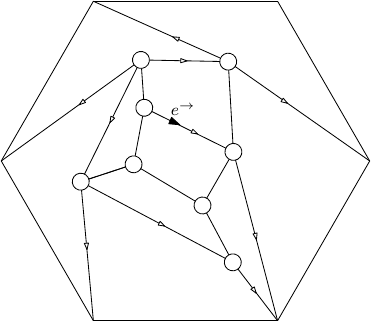}
                \caption{The induced orientation (edges without arrows are doubly oriented).}
                \label{fig:closure-orient}
\end{subfigure}%
\quad
\label{fig:closure}
\end{figure}

The map $\rM$ inherits the orientations $\{e_{\kappa}, \kappa \in \cB_B(\rT)\}$, which orients a subset of its edges (the orientation of $e^{\to}$ need not respect the inherited orientation). The remaining edges of $\rM$ are the edges of the hexagon, together with those edges of $\rT$ joining internal nodes of $\rT$, and we view all these edges as doubly oriented, or oriented in both directions (see Figure~\ref{fig:closure-orient}; in that figure the edge $e^{\to}$ is indicated by a solid black arrow, whereas the inherited orientation is shown with empty white arrows). We say a doubly oriented edge is an outgoing edge from both its endpoints. Since $\rT$ is binary, it follows that each non-hexagon node $v$ of $\rM$ has exactly three outgoing edges; such an orientation is called a {\em tri-orientation} of $\rM$. More precisely, a tri-orientation of $\rM$ is an orientation of the edges of $\rM$ (with both singly and doubly oriented edges permitted) such that every non-hexagon node has exactly three outgoing edges, and hexagon nodes have exactly two outgoing edges. 

Finally, the ``clockwise'' orientation of the closure operation straightforwardly implies that $\rM$ has no counterclockwise cycle in its interior, in the sense that any oriented cycle $C$ in $\rM$ containing at least one non-hexagon node and with its unbounded face to its right, must traverse some (singly) oriented edge of $\rM$ from head to tail. It turns out that for any irreducible quadrangulation of the hexagon there is an {\em unique} tri-orientation of $\rM$ with no counterclockwise cycle in its interior \cite[Theorem 4.4]{fusy08dissection}. Furthermore, the doubly-oriented edges of this tri-orientation form a hexagon plus a spanning tree $\rT$ of the internal vertices of $\rM$, and $\rT$ is a binary tree whose closure is $\rM$. Likewise, applying the closure operation to any tree $\rT$ yields a map $\rM$ whose ``opening'' is again $\rT$. It follows that the closure operation is a bijection. 

Let $\cT_n$ be the set of binary trees with $n$ internal nodes. 
Also, let $\cM$ be the set of pairs $(M,e^{\to})$, where $M$ is an irreducible quadrangulation of the hexagon and $e^{\to}$ is an oriented edge of $M$ at least one of whose endpoints is a non-hexagon node, and let $\cM_n=\{(M,e^{\to}) \in \cM: |v(M)|=n+6\}$. 
\begin{prop}[\cite{fusy08dissection}, Theorems 4.7 and 4.8]\label{prop:bijection}
For each $n \ge 1$, the closure operation is a bijection between $\cT_n$ and $\cM_n$. 
\end{prop}
In the next subsection, we explain the effect of ``growing'' $\rT$ -- transforming a bud into an internal node -- on the map $\rM$ resulting from the closure operation. 

\subsection{Bud growth and map growth}\label{sec:budgrowth-mapgrowth}
Given a binary tree $\rT=(T,r(T))$ and a bud corner $\kappa$ of $\rT$, {\em growing} $\rT$ at $\kappa$ means adjoining two buds incident to $v=\v(\kappa)$, so $\kappa$ becomes an internal node of degree three. Write $\rT^+$ for the resulting tree, which is still binary. Let $\rM$ and $\rM^+$ be the closures of $\rT$ and $\rT^+$, respectively. Figures~\ref{fig:growth} and~\ref{fig:growthplus} depict the corresponding difference between $\rM$ and $\rM^+$; this difference is {\em local}, in that it is confined to faces incident to $\v(\sigma_{\rT}(\kappa))$. We now explain the transformation in detail. 

\begin{figure}[ht]
\begin{subfigure}[b]{0.5\textwidth}
\begin{centering}
		\includegraphics[width=0.8\textwidth,page=1]{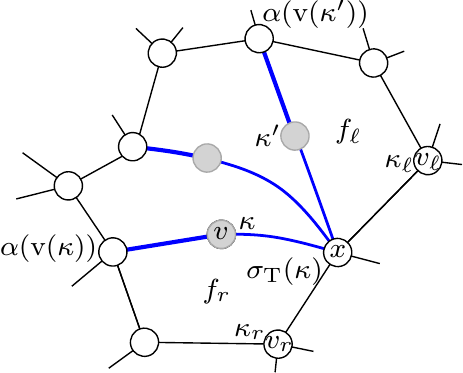}
                \caption{A part of $\rM$.}
                \label{fig:growth}
\end{centering}
\end{subfigure}%
\begin{subfigure}[b]{0.5\textwidth}
\begin{centering}
		\includegraphics[width=0.8\textwidth,page=2]{figures/hexgrow.pdf}
                \caption{The same part of $\rM^+$.}
                \label{fig:growthplus}
\end{centering}
\end{subfigure}%
\caption{Gray nodes are nodes of $\rT$ but not of $\rM$ (left), or of $\rT^+$ but not of $\rM^+$ (right). Blue edges have a thick part and a thin part, divided by a gray node; the thick part of each blue edge is itself an edge of $\rT$ (left) or $\rT^+$ (right).}
\label{fig:hex}
\end{figure}

Let $\kappa'$ be the first bud corner with $\sigma_{\rT}(\kappa')=\sigma_{\rT}(\kappa)$, in the sense that if $\kappa^*$ is any other bud corner with $\sigma_{\rT}(\kappa^*)=\sigma_{\rT}(\kappa)$ then $\kappa' \prec \kappa^* \prec \kappa$. Then let $K(\kappa) = \{\kappa^* \in \cC_B(\rT): \sigma_{\rT}(\kappa') = \sigma_{\rT}(\kappa)\}$. In Figure~\ref{fig:growth}, the nodes incident to corners in $K(\kappa)$ are precisely the greyed nodes

Write $x = \v(\sigma_{\rT}(\kappa))$, let $e_{\ell}=\{\parent(\v(\kappa')),x\}$ and let $e_r = \{\parent(\v(\kappa)),x\}$, and let $f_{\ell}$ and $f_r$ be the faces of $\rM$ lying to the left and right of $e_{\ell}$ and $e_r$, respectively. Then let $v_\ell$ and $v_r$ be the vertices of $\rM$ diagonally opposite $\parent(\v(\kappa'))$ and $\parent(\v(\kappa))$ on $f_{\ell}$ and on $f_r$, respectively, and let $\kappa_{\ell}$ and $\kappa_r$ be the corners of $f_\ell$ and $f_r$ incident to $v_\ell$ and $v_r$. Again, see Figure~\ref{fig:growth}.  

With these definitions, $\rM^{+}$ is formed from $\rM$ as follows. For all $\xi \in K(\kappa)$ remove $e_{\xi}$ from $\rM$. This creates a face of degree $2|K(\kappa)|+4$. Add a vertex $v$ in this face (recall that we also wrote $v=\v(\kappa)$; this is deliberate), and add edges from $v$ to $\parent(\v(\xi))$ for each $\xi \in K(\kappa)$, and from $v$ to $\v(\kappa_{\ell})$ and $\v(\kappa_r)$. 

To prove that this description is valid, argue as follows.\footnote{If Figure~\ref{fig:hex} is sufficiently convincing, feel free to skip straight to Section~\ref{sec:growing}.} Viewing $\rT$ as a subtree of $\rT^+$ in the natural way, for $\kappa' \in \cC(T)\setminus \{\kappa\}$ we have $S_{\rT}(\kappa')=S_{\rT^+}(\kappa')$. It follows that if $\kappa' \in \cC(T) \setminus K(\kappa)$ then $\sigma_{\rT}(\kappa') = \sigma_{\rT^+}(\kappa')$. In $\rT^+$ the corner $\kappa_1(v)$ is an internal corner, and has $S_{\rT^+}(\kappa_1(v))=S_{\rT}(\kappa)$. Furthermore, for all $\kappa'\in K(\kappa)\setminus \{\kappa\}$, in $\rT^+$ we have $\kappa' \prec \kappa_1(v) \prec \sigma_{\rT}(\kappa')=\sigma_{\rT}(\kappa)$ and thus $\sigma_{\rT^+}(\kappa')=\kappa_1(v)$. 

Next write $\xi_{\ell}$ and $\xi_r$ for the bud corners incident to the left and right children of $v$ in $\rT^+$ (these are the two greyed nodes incident to $v$ in Figure~\ref{fig:growthplus}). Then $S_{\rT^+}(\xi_{\ell})=S_{\rT}(\kappa)-1$ and $S_{\rT^+}(\xi_r)=S_{\rT}(\kappa)+1$. 
We claim that $\sigma_{\rT^+}(\xi_\ell)=\kappa_\ell$ and $\sigma_{\rT^+}(\xi_r)=\kappa_r$; proving this will establish the validity of the above description. 

In proving the above claim, it is useful to extend the domain of definition of $\sigma_{\rT}$ from $\cC_B(\rT)$ to $\cC(\rT)$ as follows. For $\kappa \in \cC_I(\rT)$, if the first corner $\kappa'$ after $\kappa$ in the counterclockwise tour around $\rT$ is internal set $\sigma_{\rT}(\kappa)=\kappa'$, and otherwise set $\sigma_{\rT}(\kappa)=\sigma_{\rT}(\kappa')$. 
We likewise define $\sigma_{\rT^+}(\kappa)$ for $\kappa \in \cC_I(\rT^+)$. 
With the above definition, for all $\kappa \in \cC_I(T)$, $S_{\rT}(\sigma_{\rT}(\kappa))=S_{\rT}(\kappa)-6\I{\kappa \prec \rho \prec \sigma_{\rT}(\kappa)}-1$, and all corners $\kappa'$ with $\kappa \prec \kappa' \prec \sigma_{\rT}(\kappa)$ have $S_{\rT}(\kappa') \ge S_{\rT}(\kappa)-6 \I{\kappa \prec \rho \preceq \kappa'}$. The analogous assertion holds for $\rT^+$. 

First consider $\xi_\ell$; we must show that $\kappa_{\ell}$ is the first corner $\kappa'$ after $\xi_{\ell}$ in $\rT^+$ with $S_{\rT^+}(\kappa') = S_{\rT^+}(\xi_\ell)-6\I{\xi_\ell \prec \rho \preceq \kappa'}= S_{\rT}(\kappa)-1-6\I{\xi_\ell \prec \rho \preceq \kappa'}$. All corners $\kappa' \in \cC(\rT^+)$ with $\xi_{\ell} \prec \kappa' \preceq \sigma_{\rT}(\kappa)$ have $S_{\rT^+}(\kappa') \ge S_{\rT}(\kappa)-6\I{\kappa \prec \rho \preceq \kappa'}$; this follows from the definition of $S_{\rT^+}$ for the corners of $\cC(\rT^+)\setminus \cC(\rT)$, and for the remaining corners follows from the definition of $\sigma_{\rT}(\kappa)$. Also, $\sigma_{\rT}(\kappa)$ is internal and $\kappa_\ell = \sigma_{\rT}(\sigma_{\rT}(\kappa))$, so all corners $\kappa'$ with $\sigma_{\rT}(\kappa) \prec \kappa' \prec \kappa_{\ell}$ have $S_{\rT^+}(\kappa') \ge S_{\rT^+}(\sigma_{\rT}(\kappa)) - 6\I{\sigma_{\rT}(\kappa)\prec \rho \preceq \kappa'}$. 
Since $\rho$ can not lie both between $\kappa$ and $\sigma_{\rT}(\kappa)$  {\em and} between $\sigma_{\rT}(\kappa)$ and $\kappa_{\ell}$, the result follows. 

The argument for $\xi_r$ is similar so we only sketch it. Let $\xi^*$ be the corner of $f_r$ incident to $\parent(\kappa)$, and note that $\xi^*$ is a corner of both $\rT$ and $\rT^+$. Thus $S_{\rT^+}(\xi^*)=S_{\rT}(\xi^*)=S_{\rT}(\kappa)+3$. Since $\kappa_r = \sigma_{\rT}(\sigma_{\rT}(\xi^*))$, by twice applying the definition of the attachment corner, $\kappa_r$ must be the first corner $\kappa'$ after $\xi_r$ in $T^+$ with $S_{\rT^+}(\kappa') \le S_{\rT}(\kappa)+1-6\I{\xi_r \prec \rho \preceq \kappa'}=S_{\rT^+}(\xi_r)-6\I{\xi_r \prec \rho \preceq \kappa'}$. 

\section{\large {\bf Growing uniformly random trees and, thus, maps}}\label{sec:growing}
In the infinite rooted binary tree $\mathbb{T}$ every node has precisely two children -- one left and one right child -- so all nodes have degree three but the root, which has degree two. Depth-$n$ nodes in $\mathbb{T}$ may be represented as strings in $\{-1,1\}^n$, with $-1$ and $1$ representing left and right. We adopt this point of view, and identify $\mathbb{T}$ with its node set $\bigcup_{i \ge 0} \{-1,1\}^{i}$, where $\{-1,1\}^0 = \{\emptyset\}$ and $\emptyset$ is the root of $\mathbb{T}$.

It is temporarily useful to view binary trees as subtrees of $\mathbb{T}$ as follows. For a given binary tree $\rT$, add a vertex $z$ in the middle of edge $e^{\to}(\rT)$. Identify $z$ with $r(\mathbb{T})$, and the head and tail of $e^{\to}$ with the left and right children of $r(\mathbb{T})$, respectively.  Recursively embed the remaining nodes by making first and second children in $\rT$ respectively corresond to left and right children in $\mathbb{T}$. See Figure~\ref{fig:treeinatree} for an illustration. The plane tree $\rT$ can be recovered from its representation as a subtree of $\mathbb{T}$ -- essentially by replacing the path from $1$ to $0$ through $\emptyset$ by a single edge -- so this viewpoint is reasonable. With this perspective we have $e^{\to}(\rT)=(1,-1)$ and $r(\rT)=1$; as usual $e^{\to}(\rT)$ is the second edge incident to the root corner $\rho$. 

The following result of Luczak and Winkler \cite{luczak04building} is key tool in the current work. 
\begin{thm}[\cite{luczak04building}, Theorem 4.1]
There exists a sequence $(\rT_n,n \ge 1)$ of random binary trees with the following properties. 
\begin{enumerate}
\item For each $n \ge 1$, $\rT_n$ is uniformly distributed in $\cT_n$.
\item For each $n \ge 1$, there is a bud corner $\xi_n$ of $T_n$ such that $T_{n+1}$ is obtained from $T_n$ by growing at $\xi_n$. In particular, the sequence is increasing so has a limit $T_{\infty} \subset \mathbb{T}$.  
\item The limit $\rT_{\infty}$ is a critical binomial Galton-Watson tree, conditioned to be infinite. 
\end{enumerate}
\end{thm}
Here is a more detailed explanation of property (3). Let $(X_n,n \ge 1)$ be iid with $\p{X_n=1}=1/2=\p{X_n=-1}$, and for each $n$ let $Y_n=-X_n$. Let $P$ be the infinite path $\{(X_1,\ldots,X_i), i \ge 0\}$ in $\mathbb{T}$. Let $(B_n,n \ge 1)$ be independent Galton-Watson trees with offspring distribution $\mu$, where $\mu(\{0\})=1/2=\mu(\{2\})$, and for each $i \ge 1$, append $B_i$ to $P$ by rooting $B_i$ at node $(X_1,\ldots,X_{i-1},Y_i)$. The resulting tree (see Figure~\ref{fig:binaryiic} for an illustration) has the law of $\rT_{\infty}$. 
\begin{figure}[ht]
\begin{subfigure}[t]{0.33\textwidth}
		\includegraphics[width=\textwidth,page=1]{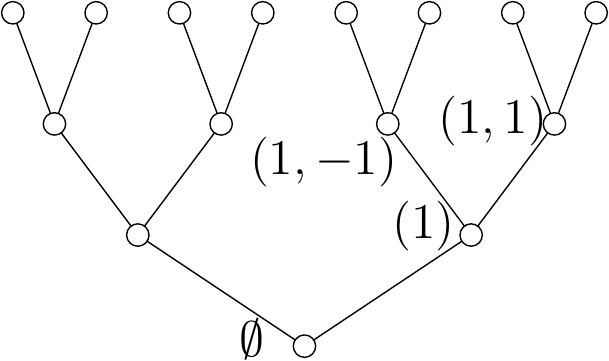}
\end{subfigure}%
\begin{subfigure}[t]{0.33\textwidth}
		\includegraphics[width=\textwidth,page=2]{figures/tree_embed.pdf}
\end{subfigure}%
\begin{subfigure}[t]{0.33\textwidth}
		\includegraphics[width=\textwidth,page=3]{figures/tree_embed.pdf}
\end{subfigure}%
\caption{{\em Left}: a part of $\mathbb{T}$, with some node labels shown. {\em Center}: a binary tree $\rT$. {\em Right}: $\rT$ viewed as a subtree of $\mathbb{T}$.}
\label{fig:treeinatree}
\end{figure}

By Proposition~\ref{prop:bijection}, the closure operation associates to each tree $\rT_n$ a map $\rM_n$ which is uniformly distributed in $\cM_n$. It therefore seems reasonable to expect that applying the closure rules to $\rT_{\infty}$ yields a map $\rM_{\infty}$ and that $\rM_n \stackrel{\mathrm{a.s.}}{\rightarrow} \rM_{\infty}$. This is indeed the case, and proving so is the subject of the remainder of the section.
\begin{wrapfigure}[15]{R}{.25\textwidth}
\rule[5pt]{.25\textwidth}{0.5pt}
                 \includegraphics[width=.25\textwidth,page=1]{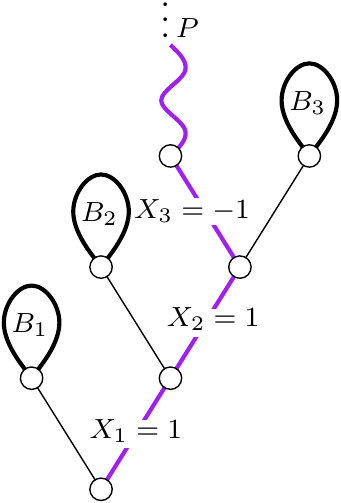}
\caption{\small A part of\ $\rT_{\infty}$, with its unique infinite path in purple.}
\label{fig:binaryiic}
\rule[8pt]{.25\textwidth}{0.5pt} 
\end{wrapfigure}

We now view $\rT_{\infty}$ as a binary plane tree (rather than as a subtree of $\mathbb{T}$). 
The set of corners $\cC(\rT_{\infty})$ consists pairs $(e,e')$, where $e'$ follows $e$ in the counterclockwise walk around (the unique, infinite face of) $\rT_{\infty}$.\footnote{We call it a walk rather than a tour since it is not closed.} Write $\prec$ for the total order on $\cC(\rT_{\infty})$ given by this walk.

The bud corners $\cC_B(\rT_{\infty})$ and internal corners $\cC_I(\rT_{\infty})$ are defined as before. Write $\cC^\ell$ and $\cC^r$ for the set of corners following and preceding $\rho$ in the counterclockwise walk around $\rT_{\infty}$, respectively. Define labels $S_{\rT_{\infty}}: \cC(\rT_{\infty}) \to \Z$ exactly as in (\ref{eq:labeldef}). The second description of the labels, given just after (\ref{eq:labeldef}) for finite trees, again applies: 
$S_{\rT_{\infty}}(\rho)=-4\I{\rho~\mathrm{is~a~bud~corner}}$, and 
in a counterclockwise walk, labels decrease by one when leaving an internal corner and increase by three when leaving a bud corner (except when the walk arrives at the root corner; then one must additionally subtract six). 

For $\kappa \in \cC_B(\rT_{\infty})$, let $\sigma(\kappa)=\sigma_{\rT_{\infty}}(\kappa)$ be the first corner $\kappa'$ following $\kappa$ in the counterclockwise walk around $\rT_{\infty}$ for which $S_{\rT_{\infty}}(\kappa') \le S_{\rT_{\infty}}(\kappa)-6\I{\kappa \prec \rho \preceq \kappa'}$, if such a corner exists. Otherwise, set $\sigma_{\rT_{\infty}}(\kappa)=-\infty$. It is immediate that if $\sigma_{\rT_{\infty}}(\kappa) \ne -\infty$ then $S_{\rT_{\infty}}(\sigma(\kappa)) = S_{\rT_{\infty}}(\kappa)-6\I{\kappa \prec \rho \preceq \sigma(\kappa)}$. The set of corners $\kappa$ with $\sigma(\kappa)=-\infty$ is the analogue of the set $\Delta$ of corners attaching to the hexagon when closing a finite binary tree. The next proposition states that this set vanishes in the $n \to \infty$ limit, and is my excuse for the paper's subtitle.

\begin{prop}
There are almost surely no corners $\kappa \in \cC(\rT_{\infty})$ with $\sigma_{\rT_{\infty}}(\kappa)=-\infty$. 
\end{prop}
\begin{proof}
For any corner $\xi \in \cC(\rT_{\infty})$, all but finitely many elements of $\cC^{\ell}$ follow $\xi$ in the walk. It thus suffices to show that $\inf\{S_{\rT_{\infty}}(\kappa): \kappa \in \rT_{\infty}\}=-\infty$. 

View $\rT_{\infty}$ as built from the path $P$, the random variables $X_i$ and the random trees $B_i$ as above, and for $n \ge 0$ let $v_n$ be the $n$'th node along $P$ (so $v_0=r(\rT_\infty)$). Then for all $n \ge 1$, $\kappa_1(v_n) \in \cC^{\ell}$ and $S_{\rT_{\infty}}(\kappa_1(v_n)) = \sum_{i=1}^n X_i$. Thus $(S_{\rT_{\infty}}(\kappa_1(v_n)),n \ge 1)$ forms a symmetric simple random walk and so $\inf\{S_{\rT_{\infty}}(\kappa_1(v_n)),n \ge 1\} = -\infty$ almost surely. 
\end{proof}
For $\kappa \in \cC_B(\rT_\infty)$ define the closure operation, the closure edge $e_{\kappa}$, and the oriented cycle $C_{\kappa}$ exactly as in Section~\ref{sec:closurebij}. By the minimality of $\sigma(\kappa)$, all corners lying $\kappa'$ within the bounded face of $C_{\kappa}$ have $S_{\rT_{\infty}}(\kappa') > S_{\rT_{\infty}}(\kappa)-6$. Since any infinite path leaving $\kappa$ follows $P$ for all but finitely many steps, and the corners along $P$ take unboundedly large negative values, it follows that the interior of $C_{\kappa}$ contains only finitely many vertices of $\rT_{\infty}$

Let $\rM_{\infty}$ be formed by closing $\kappa$ in the corner $\sigma(\kappa)$ for each $\kappa \in \cC_B(\rT_{\infty})$. 

\begin{prop}
$\rM_{\infty}$ is almost surely locally finite.
\end{prop}
\begin{proof}
First observe that for $\kappa,\xi \in \cC_B(\rT_{\infty})$, if $\kappa \prec \xi \prec \sigma(\kappa)$ then $\kappa \prec \sigma(\xi) \preceq \sigma(\kappa)$.  It follows that if $\xi \prec \kappa$ then either $\sigma(\xi) \prec \kappa$ or $\sigma(\kappa) \prec \sigma(\xi)$. 
Since $\rT_{\infty}$ is a.s.\ locally finite and the cycles $C_{\kappa}$ a.s.\ have finite interior, it suffices to show that a.s.\ for all internal corners $\xi$ of $\rT_{\infty}$, the set $\{\kappa \in \cC_B(\rT_{\infty}): \sigma(\kappa)=\xi\}$ is a.s.\ finite. Let $\pi(\xi)$ be the maximal corner $\xi'$ (with respect to $\prec$) preceding $\xi$ for which $S_{\rT_{\infty}}(\xi') < S_{\rT_{\infty}}(\xi)-6$. 

All but finitely many corners preceding $\xi$ lie in $\cC_r$, and a similar argument to that above shows that the corners along $P$ lying in $\cC_r$ take unboundedly large negative values. Since  a.s.\ only finitely many corners lie between any two corners of $\cC(\rT_{\infty})$ with respect to $\prec$, it follows that $\pi(\xi)$ is well-defined. Furthermore, $\pi(\xi)$ must be a bud corner since its $\prec$-successor has a larger label than its own. It follows that $\xi \prec \sigma(\pi(\xi))$. Thus, by the observation from the start of the first paragraph, if $\kappa$ is any bud corner with $\sigma(\kappa)=\xi$ then necessarily $\pi(\xi) \prec \kappa \prec \xi$; and there are only finitely many such corners. 
\end{proof}
\begin{cor}\label{cor:mnweak}
$\rM_{n} \stackrel{\mathrm{a.s.}}{\rightarrow} \rM_{\infty}$ as $n \to \infty$. 
\end{cor}
\begin{proof}
Recall that $\rT_{n+1}$ is obtained from $\rT_n$ by growing at corner $\xi_n$, and the description 
in Section~\ref{sec:budgrowth-mapgrowth} of how such growth transforms the associated map. It follows from this description that for $\kappa \in \cC_B(\rT_{n})$, if $\sigma_{\rT_{n+1}}(\kappa) \ne \sigma_{\rT_{n}}(\kappa)$ then $\sigma_{\rT_{n+1}}(\kappa)=(\{\parent(\v(\xi_n)),\v(\xi_n)\},\{\v(\xi_n),c_1(\v(\xi_n))\})$. Since $c_1(\v(\xi_n))$ is not a node of $\rT_n$, it follows that $\sigma_{\rT_{n+1}}(\kappa) \not \in \cC_B(\rT_n)$. 

Now, for each corner $\kappa \in \cC(\rT_{\infty})$, let $\tau(\kappa) = \inf\{m: \kappa \in \cC(\rT_m)\}$; then $\tau(\kappa)$ is almost surely finite. By the fact in the preceding paragraph, for all $n \ge \max(\tau(\kappa),\tau(\sigma_{\rT_{\infty}}(\kappa)))$, we have $\sigma_{\rT_n}(\kappa)=\sigma_{\rT_\infty}(\kappa)$. Since $\rM_{\infty}$ is a.s.\ locally finite, it follows that for any $R>0$, there is an a.s.\ finite time $n_0$ such that for all $m,n \ge n_0$, $\rM_n^{< R}$ and $\rM_m^{< R}$ are isomorphic. 
\end{proof}

\section{\large {\bf Three-connected maps}}\label{sec:3-connectedmaps}
\begin{wrapfigure}[19]{R}{.3\textwidth}
\rule[5pt]{.3\textwidth}{0.5pt}
                 \includegraphics[width=.3\textwidth,page=1]{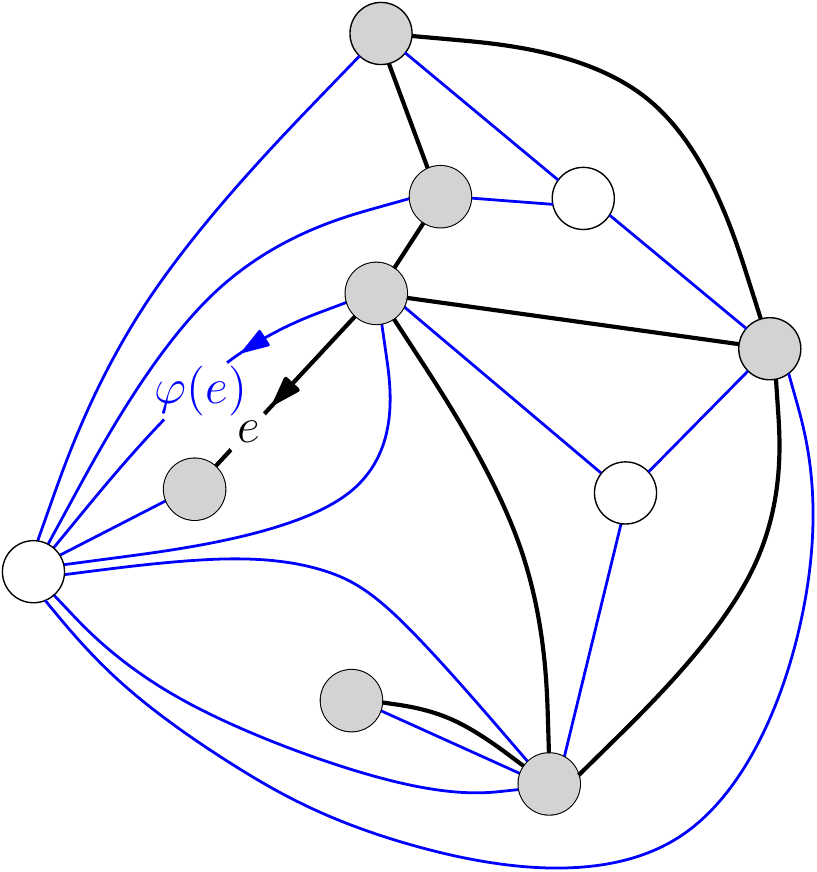}
\caption{\small A rooted map (with grey nodes and black edges) and its image under the angular mapping (with white and grey nodes and blue edges).}
\label{fig:angular}
\rule[8pt]{.3\textwidth}{0.5pt} 
\end{wrapfigure}
Given a locally finite map $G$, the {\em angular mapping} associates to $G$ a quadrangulation $Q$ as follows. Add a vertex $v_f$ in each face $f$ of $G$. For each corner $\kappa$ incident to $f$, add an edge between $v_f$  and $\v(\kappa)$ attaching in corner $\kappa$. Then erase the edges of $G$ (see Figure~\ref{fig:angular}). Here consider $G$ and $Q$ as embedded in $\mathbb{S}^2$ (in $\R^2$ there is a choice of how to draw the edges of $Q$ lying in the unbounded face of $G$), but in $\mathbb{S}^2$ this choice vanishes. Then the angular mapping is a $2$-to-$1$ map; its inverse images may be found by properly $2$-coloring the vertices of $Q$, and choosing one of the two colour classes to form the nodes of $G$. Note that the number of edges of $G$ is the number of faces of $Q$. 

There is a natural function $\varphi$ taking oriented edges of $G$ to oriented edges of $Q$: for $uv$ an oriented edge of $G$, let $f$ be the face of $G$ lying to the right of $G$, and let $\varphi(uv)=uv_f$. Note that the tail of $\varphi(uv)$ is always a node of $G$. This yields an extension of the angular mapping to rooted maps, which sends $(G,e)$ to $(Q,\varphi(e))$. The mapping is now bijective, since the orientation of $\varphi(e)$ determines which of the colour classes of $Q$ forms the nodes of $G$. 

By Proposition~\ref{prop:bijection}, for any irreducible quadrangulation of a hexagon $\rM \in \rM$, we may view $\rM$ as arising from a binary tree $\rT$ by the closure bijection. We may thus canonically label the nodes of the hexagonal face of $\rM$ with labels $0,1,2,3,4,5$. We hereafter view all $\rM \in \rM$ as endowed with such labels, and for $i \in \{0,1,2,3,4,5\}$ write $\rM(i)$ for the rooted quadrangulation obtained from $\rM_n$ by adding an edge $e^*=\{i,i+3 \mod 6\}$ from $i$ to the diagonally opposing node of the hexagon. 

Now let $(\rM_n,1 \le n \le \infty)$ be as in Section~\ref{sec:growing}. Independently of $(\rM_n,1 \le n \le \infty)$, let $U$ be uniformly distributed in $\{0,1,2,3,4,5\}$. For $1 \le n \le \infty$, let $\rQ_n=\rM_n(U)$. 
Then $\rQ_n$ is a rooted quadrangulation, and we let $\rG_n$ be the associated rooted map under the inverse of the angular mapping. 

It is well-known (see, e.g., \cite[Theorem 3.1]{fusy08dissection}) that if a quadrangulation $M$ is associated with a map $G$ under the angular mapping, then $G$ is $3$-connected if and only if $M$ is irreducible. The quadrangulation $\rQ_n$ need not be irreducible, since there may be a $3$-edge path between the endpoints of $e^*$ passing through the interior of the hexagon in $\rM_n$. However, it turns out that $\rQ_n$ {\em is} irreducible with uniformly positive probability, and hence $\rG_n$ is $3$-connected with uniformly positive probability. The main point of this section is to show (in Proposition~\ref{prop:gnlim}) that $\rG_n$ a.s.\ converges, and to identify the limit (in Theorem~\ref{thm:3con})  as ``the uniform infinite $3$-connected planar map'' (or -- thanks to Whitney's theorem -- graph). However, we first briefly describe how the growth dynamics modify $\rG_n$ (though this information is not in fact needed in the paper). 

The effect on $\rG_n$ of growing $\rT_n$ depends on whether the growing bud corner is attached to a primal or a facial vertex. If it is attached to a facial vertex then growing $\rT_n$ adds an edge within the face. If it is attached to a primal vertex then growing instead ``uncontracts'' the primal vertex, turning it into two vertices joined by an edge; this may be seen as adding an edge to the facial dual graph of $\rG_n$. In either case, this is a ``local'' modification in that it only affects the nodes that lie on a common face with a given vertex. We now turn to the main business of the section.

\begin{prop}\label{prop:gnlim}
$\rG_n \stackrel{\mathrm{a.s.}}{\rightarrow} \rG_{\infty}$ as $n \to \infty$.
\end{prop}
\begin{proof}
Fix any map $G$ and let $M$ be the image of $G$ under the angular mapping. Then for all $\{u,v\} \in e(G)$, there exists a path of length two between $u$ and $v$ in $Q$. Thus, for all $x,y \in v(G)$, $d_G(x,y) \ge d_M(x,y)/2$, where $d_G$ and $d_Q $ denote graph distance in $G$ and $Q$ respectively. Since $\rQ_n$ and $\rG_n$ have the same root node and $\rQ_n \stackrel{\mathrm{a.s.}}{\rightarrow} \rQ_{\infty}$, the result follows. 
\end{proof}
\begin{thm}\label{thm:3con}
Let $\hat{\rG}_n$ be uniformly distributed on the set of $3$-connected rooted maps with $n+4$ edges; then $\hat{\rG}_n$ converges in distribution to $\rG_{\infty}$ in the local weak sense. In particular, the $\rG_{\infty}$ is almost surely three-connected. 
\end{thm}
\begin{proof}
By Theorem~4.8 and Lemma~5.1 of \cite{fusy08dissection}, the conditional distribution of $\rG_n$, given that it is $3$-connected, is uniform in the set of $3$-connected graphs with $n+4$ edges (note that $\rQ_n$ has $n+4$ faces). Furthermore, by Proposition 6.1 of \cite{fusy08dissection}, $\p{\rG_n\mbox{ is $3$-connected}} \to 2^8/3^6$ as $n \to \infty$. 
Write $\cM(i) = \{\rM \in \cM: \rM(i)~\mbox{is~irreducible}\}$, and observe that we may then write the event that $\rG_n$ is $3$-connected as $E(\rM_n) = \{\rM_n \in \cM(U)\}$. 

We prove the theorem by showing that for fixed $R>0$, $E(\rM_n)$ is asymptotically independent of $\rM_n^{< R}$. More precisely, we show that any set $\mathcal{S}$ of finite rooted planar maps, 
\begin{equation}\label{eq:toprove}
\left|\p{M_n^{<R} \in \cS,E(\rM_n)}
-\p{M_n^{<R} \in \cS} \p{E(\rM_n)}\right| \to 0\, ,
\end{equation}
as $n \to \infty$. Since $\rM_n^{< R} \to \rM_{\infty}^{<R}$ almost surely, and $\rG_n^{<R/2}$ is determined by $\rM_n^{<R}$, the first assertion of the theorem then follows. Having established this, since $\rG_\infty$ is the distributional limit of a sequence of random $3$-connected maps, it must itself be a.s.\ $3$-connected. It thus remains to prove that $E(\rM_n)$ and $\rM_{n}^{<R}$ are indeed asymptotically independent. For the remainder of the proof, we fix $R> 0$ and a set $\cS$ as above. 

We again view $\rT_{\infty}$ as constructed from the infinite path $P$ and the random variables $X_i$ and random trees $B_i$. List the nodes of $P$ as $(v_i,i \ge 1)$. Write $V_{n}^{<R}$ for the set of nodes of $\rT_{n}$ that either (a) are nodes of $\rM_{n}^{<R}$ or (b) are buds whose attachment corner is incident to a node of $\rM_{n}^{<R}$. 

\begin{figure}[ht]
\hspace{-0.5cm}
\begin{subfigure}[t]{0.5\textwidth}
\begin{center}
		\includegraphics[width=0.66\textwidth,page=4]{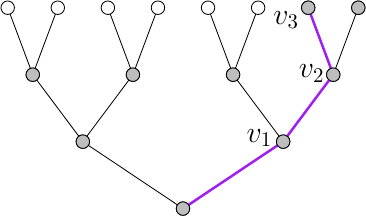}
		\end{center}
                \caption{Nodes of $\mathrm{T}_n$ are grey, edges of $P$ are purple.}
\end{subfigure}%
\begin{subfigure}[t]{0.5\textwidth}
\begin{center}
		\includegraphics[width=0.66\textwidth,page=5]{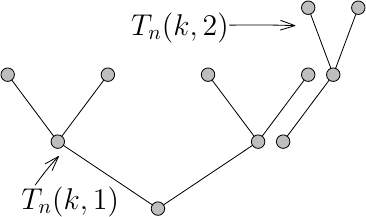}
		\end{center}
                \caption{The trees $\mathrm{T}_n(k,1)$ and $\mathrm{T}_n(k,2)$ (here $k=1$).\\}
\end{subfigure}%
\caption{Splitting $\mathrm{T}_n$ into $\mathrm{T}_n(k,1)$ and $\mathrm{T}_n(k,2)$.}
\label{fig:treesplit}
\end{figure}
Given $k > 1$, let $\rT_n(k,1)$ be the subtree of $\rT_n$ containing the root edge when all strict descendants of $v_{k+1}$ are removed, and let $\rT_n(k,2)$ be the subtree of $\rT_n$ consisting of $v_{k}$ and all its strict descendants; see Figure~\ref{fig:treesplit}. (If $v_{k+1}$ is not a node of $\rT_n$ we agree that $\rT_n(k,2)$ is empty.) Since $\rT_n$ is an uniformly random binary tree, conditional on its size $\rT_n(k,2)$ is an uniformly random binary tree and is independent of $\rT_n(k,1)$. 

Let $H_n$ be the set of nodes of $\rT_n$ incident to a corner $\kappa$ with $S_{\rT_n}(\kappa) \le s^*(\rT_n)+12$. For any three-edge path in $\rM_n$ joining distinct vertices of the hexagon, if the internal nodes of the path do not lie on the hexagon then they are nodes of $\rT_n$ that neighbour buds which attach to the hexagon. Since buds $\kappa$ attaching to the hexagon have $S_{\rT_n}(\kappa) \le s^*(\rT_n)+11$, it follows that all nodes of such a path belong to $H_n$. 

Write $\rM_n(k,1)$ and $\rM_n(k,2)$ for the closures of $\rT_n(k,1)$ and $\rT_n(k,2)$. Then conditional on its size, $\rM_n(k,2)$ is an uniformly random quadrangulation of the hexagon. 
Furthermore, by reasoning similar to that in the proof of Corollary~\ref{cor:mnweak}, it is straightforward to see that for all $\kappa \in \cC_B(\rT_n)$, if $\kappa$ and $\sigma_{\rT_n}(\kappa)$ are both elements of $\cC_B(\rT_n(k,1))$ then $\sigma_{\rT_n}(\kappa)=\sigma_{\rT_n(k,1)}(\kappa)$. Likewise, if $\kappa,\sigma_{\rT_n}(\kappa) \in \cC_B(\rT_n(k,2))$ then $\sigma_{\rT_n}(\kappa)=\sigma_{\rT_n(k,2)}(\kappa)$. 

Now let $A(k,n)$ be the event  that no node of $V_{n}^{<R}$ is a weak descendant of $v_k$, that $H_n \subset v(\rT_n(k,2))$, and that $|v(\rT_n(k,2))| \ge n/2$. 
Almost surely, $V_{n}^{<R}=V_{\infty}^{<R}$ for all $n$ sufficiently large, and $s^*(\rT_n)$ a.s.\ decreases to $-\infty$. For any $\eps > 0$, we may therefore choose $k$ and $n$ large enough that $\p{A(k,n)} > 1-\eps$, and for such $k$ and $n$ we have 
\[
|\p{\rM_n^{<R}\in \mathcal{S},E(\rM_n)} - \p{\rM_n^{<R}\in \mathcal{S},E(\rM_n),A(k,n)}| < \eps\, . 
\]
Furthermore, by the above observations about consistency of closure locations in $\rT_n$ and $\rT_n(k,2)$, if  $A(k,n)$ occurs then $\rM_n^{<R} = \rM_n^{<R}(k,1)$ and $E(\rM_n)=E(\rM_n(k,2))$, where we let $E(\rM_n(k,2)) = \{\rM_n(k,2) \in \cM(U)\}$. 
We thus have 
\begin{align*}
& \p{\rM_n^{<R}\in \mathcal{S},E(\rM_n),A(k,n)} \\
= & \p{\rM_n^{<R}(k,1)\in \mathcal{S},E(\rM_n(k,2)),A(k,n)} \\
= & \p{\rM_n^{<R}(k,1)\in \mathcal{S},A(k,n)}\p{E(\rM_n(k,2))~|~A(k,n),\rM_n^{<R}(k,1)\in \mathcal{S}} \, \, 
\end{align*}
so 
\begin{align*}
&  |\p{\rM_n^{<R}\in \mathcal{S},E(\rM_n)}  & \\
 &- \p{\rM_n^{<R}(k,1)\in \mathcal{S},A(k,n)}\cdot \p{E(\rM_n(k,2))|A(k,n),\rM_n^{<R}(k,1)\in \mathcal{S}}|  < \eps. 
\end{align*}
For $k$ and $n$ large enough that $\p{A(k,n)} > 1-\eps$, we also have  
\begin{align*}
|\p{\rM_n^{<R}\in \mathcal{S}} -\p{\rM_n^{<R}(k,1)\in \mathcal{S},A(k,n)}| & < \eps. 
\end{align*}

Now recall that $\rT_n(k,2)$ is an uniform binary tree and is independent of $\rT_n(k,1)$ conditional on its size. It follows that given that $|v(\rT_n(k,2))|=m$, $\rM_n(k,2)$ is distributed as $\rM_m$, so 
\[
\p{E(\rM_n(k,2))~|~A(k,n),\rM_n^{<R}(k,1)\in \mathcal{S},|v(\rT_n(k,2))|=m} = \p{E(\rM_m)}\, .
\]
Given $A(k,n)$ we have $n/2 \le |v(\rT_n(k,2))| \le n$, so by the triangle inequality 
\[
|\p{E(\rM_n(k,2))~|~A(k,n),\rM_n^{<R}(k,1)\in \mathcal{S}} - \p{E(\rM_n)}| 
\le \sup_{n/2 \le m \le n} 2|\p{E(\rM_m)}-2^8/3^6|\, ,
\]
which is also less than $\eps$ for $n$ sufficiently large since $\p{E(\rM_n)} = \p{\rG_n\mbox{ is $3$-connected}} \to 2^8/3^6$ as $n \to \infty$. 
The preceding inequalities (and the fact that probabilities lie between zero and one) then yield that for $n$ large, 
\[
\left|\p{M_n^{<R} \in \cS,E(\rM_n)}
-\p{M_n^{<R} \in \cS} \p{E(\rM_n)}\right| < 2\eps\, .
\]
Since $\eps$ was arbitrary, this establishes (\ref{eq:toprove}).
\end{proof}
\section{\large {\bf Questions and remarks}}\label{sec:conc}
\begin{enumerate}
\setcounter{enumi}{-1}
\item It is not hard to show using enumerative results for irreducible quadrangulations with boundary \cite{mullin68quad} that the degree of the root node has exponential tails in both $M_n$ and $G_n$.\footnote{For the assiduous reader: what we call irreducible was called {\em simple} in \cite{mullin68quad}.} It then follows, from the general result of Gurel-Gurevich and Nachmias \cite{gg13recurrence} on recurrence of planar graph limits, that simple random walk is recurrent on both graphs. 
\item A pioneering work of Brooks, Smith, Stone and Tutte \cite{brooks40dissection} showed how to associate to a squaring of a rectangle with any rooted $3$-connected map. For random maps, this yields a random squaring. We believe it is possible to show that when appropriately rescaled to remain compact, the squarings corresponding to the sequence $G_n$, viewed as random subsets of $\R^2$ converge almost surely for the Hausdorff distance. We are currently pursuing this line of enquiry \cite{ab14square}. 
\item Rather than rooting at $e^{\to}$, it would also be natural to re-root $Q_n$ at (an orientation of) the edge $e^*$ that is added to the hexagon. However, it is straightforward to see that the sequence $(Q_n,e^*)$ can not converge almost surely. This is because the hexagon essentially corresponds to the corner of minimum label in $\rT_n$, and this minimum label tends to $-\infty$ as $n \to \infty$. Nonetheless, we expect that $(Q_n,e^*)$ does converge in distribution, that the law of the limit can be explicitly described using the results from \cite{fusy08dissection}, and that this law is mutually absolutely continuous with that of $\rQ_\infty$. 
\item Marckert \cite{marckert04} essentially establishes that the contour process of $\rT_n$ and the label process $S_{\rT_n}$ (in contour order) converge jointly, after appropriate normalization, to a pair $({\bf e},Z)$, where ${\bf e}$ is a standard Brownian excursion, and $Z$ is the Brownian snake indexed by {\bf e}. (See \cite{legall12scaling} for more details on these objects and their connections with random maps.) Given this, it does not seem out of reach of current technology to prove that $\rQ_n$ converges, after renormalization, to the Brownian map.
\item There is a standard bijection between binary trees and plane trees, that consists of contracting edges from parents to right children in the binary tree; applying this to $\rT_n$ yields a plane tree $\rP_n$. Augmenting the edges of $\rP_n$ with independent uniform $\{-1,0,+1\}$ random variables and applying the Schaeffer bijection then yields an uniformly random rooted quadrangulation (not necessarily irreducible) with $n+1$ nodes. This yields a growth procedure for uniformly random quadrangulations, which I believe deserves investigation. 
\item Evans, Gr\"ubel and Wakolbinger \cite{evans12trickledown} have initiated the study of the Doob-Martin boundary of Luczak and Winkler's tree growth process. Roughly speaking, the Doob-Martin boundary corresponds to the ways in which it is possible to condition on $\rT_{\infty}$ to obtain a well-defined conditional growth procedure for $(\rT_n,n \ge 1)$. It would be interesting to revisit the growth of $\rM_n$ and $\rG_n$ in this context. 
\end{enumerate}

\section{\large {\bf Acknowledgements}}
My sincere thanks to Nicholas Leavitt, for stimulating conversations and for his helpful comments on a draft of this work. Thanks also to Marie Albenque and Gilles Schaeffer, who were the first to teach me about the bijection between $\rT_n$ and $\rM_n$, in 2011.

\vspace{-0.5cm}             
\appendix

\end{document}